\documentclass[11pt,draft]{amsart}

\usepackage{a4,amssymb,cite,epic}

\DeclareMathOperator{\Nil}{Nil}

\DeclareMathOperator{\var}{var}

\DeclareMathOperator{\ZR}{ZR}

\newtheorem{theorem}{Theorem}

\newtheorem{proposition}[theorem]{Proposition}

\newtheorem{lemma}[theorem]{Lemma}

\newtheorem{corollary}[theorem]{Corollary}

\makeatletter

\renewcommand*\subjclass[2][2000]{\def\@subjclass{#2}\@ifundefined
{subjclassname@#1}{\ClassWarning{\@classname}{Unknown edition (#1) of
Mathematics Subject Classification; using '2000'.}}{\@xp\let\@xp
\subjclassname\csname subjclassname@#1\endcsname}}

\makeatother

\begin{document}

\title[Definability of the variety generated by a commutative monoid]
{Definability of the variety\\
generated by a commutative monoid\\
in the lattice of commutative\\
semigroup varieties}

\author{B. M. Vernikov}

\address{Department of Mathematics and Mechanics, Ural State University,
Lenina 51, 620083 Ekaterinburg, Russia}

\email{bvernikov@gmail.com}

\date{}

\thanks{The work was partially supported by the Russian Foundation for Basic
Research (grants No.\,09-01-12142,\,10-01-000524) and the Federal Education
Agency of the Russian Federation (project No.\,2.1.1/3537).}

\begin{abstract}
Let $M$ be a commutative monoid. We construct a first-order formula that
defines the variety generated by $M$ in the lattice of all commutative
semigroup varieties.
\end{abstract}



\maketitle

A subset $A$ of a lattice $\langle L;\vee,\wedge\rangle$ is called \emph
{definable in} $L$ if there exists a first-order formula $\Phi(x)$ with one
free variable $x$ in the language of lattice operations $\vee$ and $\wedge$
which \emph{defines $A$ in} $L$. This means that, for an element $a\in L$,
the sentence $\Phi(a)$ is true if and only if $a\in A$. If $A$ consists of a
single element, we speak about definability of this element.

We denote the lattice of all commutative semigroup varieties by \textbf{Com}.
A set of commutative semigroup varieties $X$ (or a single commutative
semigroup variety $\mathcal X$) is said to be \emph{definable} if it is
definable in \textbf{Com}. In this situation we will say that the
corresponding first-order formula \emph{defines} the set $X$ or the variety
$\mathcal X$.

Let $M$ be a commutative monoid. In~\cite[Corollary~4.8]{Vernikov-def}, we
provide an explicit first-order formula that defines the variety generated by
$M$ in the lattice of all semigroup varieties. The objective of this note is
to modify the arguments from~\cite{Vernikov-def} in order to present an
explicit formula that defines the variety generated by $M$ in the lattice
\textbf{Com}.

We will denote the conjunction by \& rather than $\wedge$ because the latter
symbol stands for the meet in a lattice. Since the disjunction and the join
in a lattice are denoted usually by the same symbol $\vee$, we use this
symbol for the join and denote the disjunction by \textsc{or}. Evidently, the
relations $\le$, $\ge$, $<$ and $>$ in a lattice $L$ can
be expressed in terms of, say, meet operation $\wedge$ in $L$. So, we will
freely use these four relations in formulas.
Let $\Phi(x)$ be a first-order formula. For the sake of brevity, we put
$$\min\nolimits_x\bigl\{\Phi(x)\bigr\}\;\rightleftharpoons\;\Phi(x)\,\&\,
(\forall y)\,\bigl(y<x\longrightarrow\neg\Phi(y)\bigr)\ldotp$$
Clearly, the formula $\min_x\bigl\{\Phi(x)\bigr\}$ defines the set of all
minimal elements of the set defined by the formula $\Phi(x)$.

Many important sets of semigroup varieties admit a characterization in the
language of atoms of the lattice \textbf{Com}. The set of all atoms of a
lattice $L$ with 0 is defined by the formula
$$\mathtt A(x)\;\rightleftharpoons\;(\exists y)\;\bigl((\forall z)\,(y\le
z)\,\&\,\min\nolimits_x\{x\ne y\}\bigr)\ldotp$$
A description of all atoms of the lattice \textbf{Com} directly follows from
the well-known description of atoms of the lattice of all semigroup varieties
(see \cite{Evans-71,Shevrin-Vernikov-Volkov-09}, for instance). To list
these varieties, we need some notation.

By $\var\Sigma$ we denote the semigroup variety given by the identity system
$\Sigma$. A pair of identities $wx=xw=w$ where the letter $x$ does not occur
in the word $w$ is usually written as the symbolic identity $w=0$\footnote
{This notation is justified because a semigroup with such identities has a
zero element and all values of the word $w$ in this semigroup are equal to
zero.}. Let us fix notation for several semigroup varieties:
\begin{align*}
&\mathcal A_n=\var\,\{x^ny=y,\,xy=yx\}\ \text{--- the variety of Abelian
groups}\\
&\phantom{\mathcal A_n=\var\,\{x^ny=y,\,xy=yx\}\ \text{--- }}\text{whose
exponent divides}\ n,\\
&\mathcal{SL}=\var\,\{x^2=x,\,xy=yx\}\ \text{--- the variety of
semilattices},\\
&\mathcal{ZM}=\var\,\{xy=0\}\ \text{--- the variety of null semigroups}\ldotp
\end{align*}

\begin{lemma}
\label{atoms}
The varieties $\mathcal A_p$ \textup(where $p$ is a prime number\textup),
$\mathcal{SL}$, $\mathcal{ZM}$ and only they are atoms of the lattice
$\mathbf{Com}$.\qed
\end{lemma}

Put
$$\mathtt{Neut}(x)\;\rightleftharpoons\;(\forall y,z)\;\bigl((x\vee y)\wedge
(y\vee z)\wedge(z\vee x)=(x\wedge y)\vee(y\wedge z)\vee(z\wedge x)\bigr)
\ldotp$$
An element $x$ of a lattice $L$ such that the sentence $\mathtt{Neut}(x)$ is
true is called \emph{neutral}. We denote by $\mathcal T$ the trivial
semigroup variety.

\begin{lemma}[\!\!{\mdseries\cite[Theorem 1.2]{Shaprynskii-dnel}}]
\label{neutral}
A commutative semigroup variety $\mathcal V$ is a neutral element of the
lattice $\mathbf{Com}$ if and only if either $\mathcal{V=COM}$ or $\mathcal
{V=M\vee N}$ where $\mathcal M$ is one of the varieties $\mathcal T$ or
$\mathcal{SL}$, while the variety $\mathcal N$ satisfies the identity $x^2y=
0$.\qed
\end{lemma}

For convenience of references, we formulate the following immediate
consequence of Lemmas~\ref{atoms} and~\ref{neutral}.

\begin{corollary}
\label{neutral atoms}
An atom of the lattice $\mathbf{Com}$ is a neutral element of this lattice if
and only if it coincides with one of the varieties $\mathcal{SL}$ or
$\mathcal{ZM}$.\qed
\end{corollary}

A semigroup variety $\mathcal V$ is called \emph{chain} if the subvariety
lattice of $\mathcal V$ is a chain. Clearly, each atom of \textbf{Com} is a
chain variety. The set of all chain varieties is definable by the formula
$$\mathtt{Ch}(x)\;\rightleftharpoons\;(\forall y,z)\,(y\le x\,\&\,z\le x
\longrightarrow y\le z\ \text{\textsc{or}}\ z\le y)\ldotp$$

We adopt the usual agreement that an adjective indicating a property shared
by all semigroups of a given variety is applied to the variety itself; the
expressions like ``group variety'', ``periodic variety'', ``nil-variety''
etc.\ are understood in this sense.

Put
\begin{align*}
&\mathcal N_k=\var\,\{x^2=x_1x_2\cdots x_k=0,\,xy=yx\}\ (k\ \text{is a
natural number}),\\
&\mathcal N_\omega=\var\,\{x^2=0,\,xy=yx\},\\
&\mathcal N_3^c=\var\,\{xyz=0,\,xy=yx\}
\end{align*}
(in particular $\mathcal N_1=\mathcal T$ and $\mathcal N_2=\mathcal{ZM}$).
The lattice of all Abelian periodic group varieties is evidently isomorphic
to the lattice of natural numbers ordered by divisibility. This readily
implies that non-trivial chain Abelian group varieties are varieties
$\mathcal A_{p^k}$ with prime $p$ and natural $k$, and only they. Combining
this observation with results of \cite{Sukhanov-82}, we have the following

\begin{lemma}
\label{chain var}
The varieties $\mathcal A_{p^k}$ with prime $p$ and natural $k$, $\mathcal
{SL}$, $\mathcal N_k$, $\mathcal N_\omega$, $\mathcal N_3^c$ and only they
are chain varieties of commutative semigroups.\qed
\end{lemma}

Fig.\ \ref{chain varieties} shows the relative location of chain varieties
in the lattice \textbf{Com}.

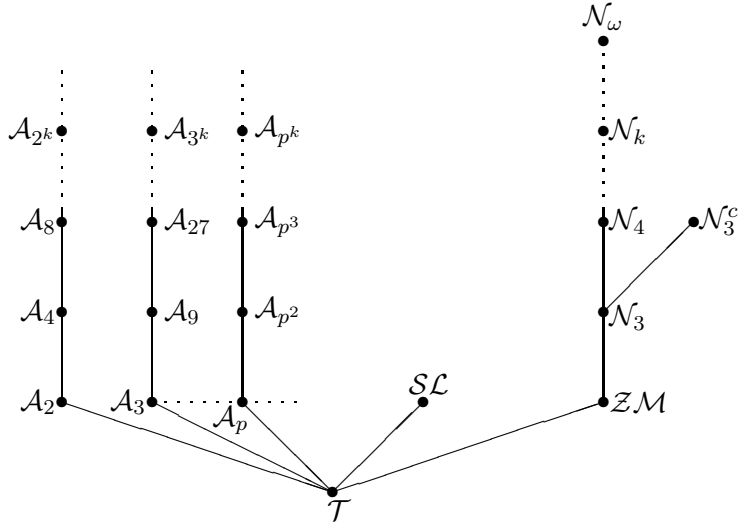
\begin{figure}[tbh]
\begin{center}
\unitlength=.8mm
\linethickness{0.4pt}
\begin{picture}(107,82)
\put(1,18){\line(0,1){32}}
\put(1,18){\line(3,-1){45}}
\put(16,18){\line(0,1){32}}
\put(16,18){\line(2,-1){30}}
\put(31,18){\line(0,1){32}}
\put(31,18){\line(1,-1){15}}
\put(46,3){\line(1,1){15}}
\put(46,3){\line(3,1){45}}
\put(91,18){\line(0,1){32}}
\put(91,33){\line(1,1){15}}
\dashline{.5}(1,50)(1,73)
\dashline{.5}(16,50)(16,73)
\dashline{.5}(18,18)(40,18)
\dashline{.5}(31,50)(31,73)
\dashline{.5}(91,50)(91,76)
\put(1,18){\circle*{1.62}}
\put(1,33){\circle*{1.62}}
\put(1,48){\circle*{1.62}}
\put(1,63){\circle*{1.62}}
\put(16,18){\circle*{1.62}}
\put(16,33){\circle*{1.62}}
\put(16,48){\circle*{1.62}}
\put(16,63){\circle*{1.62}}
\put(31,18){\circle*{1.62}}
\put(31,33){\circle*{1.62}}
\put(31,48){\circle*{1.62}}
\put(31,63){\circle*{1.62}}
\put(46,3){\circle*{1.62}}
\put(61,18){\circle*{1.62}}
\put(91,18){\circle*{1.62}}
\put(91,33){\circle*{1.62}}
\put(91,48){\circle*{1.62}}
\put(91,63){\circle*{1.62}}
\put(91,78){\circle*{1.62}}
\put(106,48){\circle*{1.62}}
\put(0,18){\makebox(0,0)[rc]{$\mathcal A_2$}}
\put(0,33){\makebox(0,0)[rc]{$\mathcal A_4$}}
\put(0,48){\makebox(0,0)[rc]{$\mathcal A_8$}}
\put(0,63){\makebox(0,0)[rc]{$\mathcal A_{2^k}$}}
\put(12,18){\makebox(0,0)[cc]{$\mathcal A_3$}}
\put(18,33){\makebox(0,0)[lc]{$\mathcal A_9$}}
\put(18,48){\makebox(0,0)[lc]{$\mathcal A_{27}$}}
\put(18,63){\makebox(0,0)[lc]{$\mathcal A_{3^k}$}}
\put(29,15){\makebox(0,0)[cc]{$\mathcal A_p$}}
\put(33,33){\makebox(0,0)[lc]{$\mathcal A_{p^2}$}}
\put(33,48){\makebox(0,0)[lc]{$\mathcal A_{p^3}$}}
\put(33,63){\makebox(0,0)[lc]{$\mathcal A_{p^k}$}}
\put(92,32){\makebox(0,0)[lc]{$\mathcal N_3$}}
\put(107,48){\makebox(0,0)[lc]{$\mathcal N_3^c$}}
\put(92,48){\makebox(0,0)[lc]{$\mathcal N_4$}}
\put(92,63){\makebox(0,0)[lc]{$\mathcal N_k$}}
\put(91,82){\makebox(0,0)[cc]{$\mathcal N_\omega$}}
\put(62,21){\makebox(0,0)[cc]{$\mathcal{SL}$}}
\put(47,0){\makebox(0,0)[cc]{$\mathcal T$}}
\put(92,18){\makebox(0,0)[lc]{$\mathcal{ZM}$}}
\end{picture}
\caption{Chain varieties of commutative semigroups}
\label{chain varieties}
\end{center}
\end{figure}

Combining above observations, it is easy to verify the following

\begin{proposition}
\label{definable sets of atoms}
The set of varieties $\{\mathcal A_p\mid p$ is a prime number$\}$ and the
varieties $\mathcal{SL}$ and $\mathcal{ZM}$ are definable.
\end{proposition}

\begin{proof}
By Lemma \ref{atoms}, all varieties mentioned in the proposition are atoms of
\textbf{Com}. By Corollary \ref{neutral atoms}, the varieties $\mathcal{SL}$
and $\mathcal{ZM}$ are neutral elements in \textbf{Com}, while $\mathcal A_p$
is not. Fig.\ \ref{chain varieties} shows that the varieties $\mathcal{ZM}$
and $\mathcal A_p$ are proper subvarieties of some chain varieties, while
$\mathcal{SL}$ is not. Therefore the formulas
\begin{align*}
&\mathtt{SL}(x)\;\rightleftharpoons\;\mathtt A(x)\,\&\,\mathtt{Neut}(x)\,\&\,
(\forall y)\,\bigl(\mathtt{Ch}(y)\,\&\,x\le y\longrightarrow x=y
\bigr),\\
&\mathtt{ZM}(x)\;\rightleftharpoons\;\mathtt A(x)\,\&\,\mathtt{Neut}(x)\,\&\,
(\exists y)\,\bigl(\mathtt{Ch}(y)\,\&\,x<y\bigr)
\end{align*}
define the varieties $\mathcal{SL}$ and $\mathcal{ZM}$ respectively, while
the the formula
$$\mathtt{GrA}(x)\;\rightleftharpoons\;\mathtt A(x)\,\&\,\neg\mathtt{Neut}
(x)\,\&\,(\exists y)\,\bigl(\mathtt{Ch}(y)\,\&\,x<y\bigr)$$
define the set $\{\mathcal A_p\mid p$ is a prime number$\}$.
\end{proof}

Note that in fact each of the group atoms $\mathcal A_p$ is individually
definable (see Proposition \ref{A_n} below). The definability of the
varieties $\mathcal{SL}$ and $\mathcal{ZM}$ is mentioned in \cite[Proposition
3.1]{Kisielewicz-04} without any explicitly written formulas.

Recall that a semigroup variety is called \emph{combinatorial} if all its
groups are trivial.

\begin{proposition}
\label{gr,comb,nil}
The sets of all Abelian periodic group varieties, all combinatorial
commutative varieties and of all commutative nil-varieties of semigroups are
definable.
\end{proposition}

\begin{proof}
It is well known that a commutative semigroup variety is an Abelian periodic
group variety [a combinatorial variety, a nil-variety] if and only if it does
not contain the varieties $\mathcal{SL}$ and $\mathcal{ZM}$ [respectively,
the varieties $\mathcal A_p$ for all prime $p$, any atoms except
$\mathcal{ZM}$]. Therefore, the sets of all Abelian periodic group varieties,
all combinatorial commutative varieties and of all commutative nil-varieties
are definable by the formulas
\begin{align*}
&\mathtt{Gr}(x)\;\rightleftharpoons\;(\forall y)\,\bigl(\mathtt A(y)\,\&\,y
\le x\longrightarrow\mathtt{GrA}(y)\bigr);\\
&\mathtt{Comb}(x)\;\rightleftharpoons\;(\forall y)\,\bigl(\mathtt A(y)\,\&\,y
\le x\longrightarrow\neg\mathtt{GrA}(y)\bigr);\\
&\mathtt{Nil}(x)\;\rightleftharpoons\;(\forall y)\,\bigl(\mathtt A(y)\,\&\,y
\le x\longrightarrow\mathtt{ZM}(y)\bigr)
\end{align*}
respectively.
\end{proof}

The claim that the set of all Abelian periodic group varieties is definable
in \textbf{Com} is proved in \cite{Kisielewicz-04} without any explicitly
written formula defining this class.

Identities of the form $w=0$ are called 0-\emph{reduced}. We denote
by $\mathcal{COM}$ the variety of all commutative semigroups. A commutative
semigroup variety is called 0-\emph{reduced in} \textbf{Com} if it is given
within $\mathcal{COM}$ by 0-reduced identities only.

\begin{proposition}
\label{0-red}
The set of all \textup0-reduced in $\mathbf{Com}$ commutative semigroup
varieties is definable.
\end{proposition}

\begin{proof}
Put
$$\mathtt{LMod}(x)\;\rightleftharpoons\;(\forall y,z)\,\bigl(x\le y
\longrightarrow x\vee(y\wedge z)=y\wedge(x\vee z)\bigr)\ldotp$$
An element $x$ of a lattice $L$ such that the sentence $\mathtt{LMod}(x)$ is
true is called \emph{lower-modular}. Lower-modular elements of the lattice
\textbf{Com} are completely determined in \cite[Theorem 1.6]
{Shaprynskii-mod&lmod}. This result immediately implies that a commutative
nil-variety is lower-modular in \textbf{Com} if and only if it is 0-reduced
in \textbf{Com}. Therefore the formula
$$\mathtt{0\text{-}red}(x)\;\rightleftharpoons\;\mathtt{Nil}(x)\,\&\,\mathtt
{LMod}(x)$$
defines the set of all 0-reduced in \textbf{Com} varieties.
\end{proof}

The following general fact will be used in what follows.

\begin{lemma}
\label{chain set}
If a countably infinite subset $S$ of a lattice $L$ is definable in $L$ and
forms a chain isomorphic to the chain of natural numbers under the order
relation in $L$ then every member of this set is definable in $L$.
\end{lemma}

\begin{proof}
Let $S=\{s_n\mid n\in\mathbb N\}$, $s_1<s_2<\cdots<s_n<\cdots$ and let $\Phi
(x)$ be the formula defining $S$ in $L$. We are going to prove the
definability of the element $s_n$ for each $n$ by induction on $n$. The
induction base is evident because the element $s_1$ is definable by the
formula $\min_x\bigl\{\Phi(x)\bigr\}$. Assume now that $n>1$ and the element
$s_{n-1}$ is definable by some formula $\Psi(x)$. Then the formula
$$\min\nolimits_x\bigl\{\Phi(x)\,\&\,(\exists y)\,\bigl(\Psi(y)\,\&\,y<x\bigr
)\bigr\}$$
defines the element $s_n$.
\end{proof}

The following lemma is a part of the semigroup folklore. It is known at least
since earlier 1980's (see \cite{Korjakov-82}, for instance). In any case, it
immediately follows from Lemma 2 of \cite{Volkov-89} and the proof of
Proposition 1 of the same article.

\begin{lemma}
\label{K+N}
If $\mathcal V$ is a commutative semigroup variety with $\mathcal{V\ne COM}$
then $\mathcal{V=K\vee N}$ where $\mathcal K$ is a variety generated by a
monoid, while $\mathcal N$ is a nil-variety.\qed
\end{lemma}

Let $C_{m,1}$ denote the cyclic monoid $\langle a\mid a^m=a^{m+1}\rangle$ and
let $\mathcal C_m$ be the variety generated by $C_{m,1}$. It is clear that
$$\mathcal C_m=\var\,\{x^m=x^{m+1},\,xy=yx\}\ldotp$$
In particular, $C_{1,1}$ is the 2-element semilattice and $\mathcal C_1=
\mathcal{SL}$. For notation convenience we put also $\mathcal C_0=\mathcal
T$. The following lemma can be easily extracted from the results of \cite
{Head-68}.

\begin{lemma}
\label{comm monoid}
If a periodic semigroup variety $\mathcal V$ is generated by a commutative
monoid then $\mathcal{V=G\vee C}_m$ for some Abelian periodic group variety
$\mathcal G$ and some $m\ge0$.\qed
\end{lemma}

Lemmas \ref{K+N} and \ref{comm monoid} immediately imply

\begin{corollary}
\label{comb and commut}
If $\mathcal V$ is a commutative combinatorial semigroup variety then
$\mathcal{V=C}_m\vee\mathcal N$ for some $m\ge0$ and some nil-variety
$\mathcal N$.\qed
\end{corollary}

Let now $\mathcal V$ be a commutative semigroup variety with $\mathcal{V\ne
COM}$. Lemmas \ref{K+N} and \ref{comm monoid} imply that $\mathcal{V=G\vee
C}_m\vee\mathcal N$ for some Abelian periodic group variety $\mathcal G$,
some $m\ge0$ and some commutative nil-variety $\mathcal N$. Our aim now is to
provide formulas defining the varieties $\mathcal G$ and $\mathcal C_m$.

It is well known that each periodic semigroup variety $\mathcal X$ contains
its greatest nil-subvariety. We denote this subvariety by $\Nil(\mathcal X)$.
Put
$$\mathcal D_m=\Nil(\mathcal C_m)=\var\,\{x^m=0,\,xy=yx\}$$
for every natural $m$. In particular, $\mathcal D_1=\mathcal T$ and $\mathcal
D_2=\mathcal N_\omega$.

\begin{proposition}
\label{C_m}
For each $m\ge0$, the variety $\mathcal C_m$ is definable.
\end{proposition}

\begin{proof}
First, we are going to verify that the formula
$$\mathtt{All\text{-}C_m}(x)\;\rightleftharpoons\;\mathtt{Comb}(x)\,\&\,
(\forall y,z)\,\bigl(\mathtt{Nil}(y)\,\&\,x=y\vee z\longrightarrow x=z\bigr
)$$
defines the set of varieties $\{\mathcal C_m\mid m\ge0\}$ in \textbf{Com}.
Let $\mathcal V$ be a commutative semigroup variety such that the sentence
$\mathtt{All\text{-}C_m}(\mathcal V)$ is true. Then $\mathcal V$ is
combinatorial. Now Corollary \ref{comb and commut} successfully applies with
the conclusion that $\mathcal{M=C}_m\vee\mathcal N$ for some $m\ge0$ and some
commutative nil-variety $\mathcal N$. The fact that the sentence $\mathtt
{All\text{-}C_m}(\mathcal V)$ is true shows that $\mathcal{M=C}_m$.

Let now $m\ge0$. We aim to verify that the sentence $\mathtt{All\text{-}C_m}
(\mathcal C_m)$ is true. It is evident that the variety $\mathcal C_m$ is
combinatorial. Suppose that $\mathcal C_m=\mathcal{M\vee N}$ where $\mathcal
N$ is a nil-variety. It remains to check that $\mathcal{N\subseteq M}$. We
may assume without any loss that $\mathcal{N=\Nil(C}_m)=\mathcal D_m$. It is
clear that $\mathcal M$ is a commutative and combinatorial variety. Corollary
\ref{comb and commut} implies that $\mathcal{M=C}_r\vee\mathcal N'$ for some
$r\ge0$ and some nil-variety $\mathcal N'$. Then $\mathcal{N'\subseteq\Nil
(C}_m)=\mathcal N$, whence
$$\mathcal C_m=\mathcal{M\vee N=C}_r\vee\mathcal{N'\vee\mathcal N=C}_r\vee
\mathcal N\ldotp$$
It suffices to prove that $\mathcal{N\subseteq C}_r$ because $\mathcal{N
\subseteq C}_r\vee\mathcal{N'=M}$ in this case. The equality $\mathcal C_m=
\mathcal C_r\vee\mathcal N$ implies that $\mathcal C_r\subseteq\mathcal C_m$,
whence $r\le m$. If $r=m$ then $\mathcal{N\subseteq C}_r$, and we are done.
Let now $r<m$. Then the variety $\mathcal C_m=\mathcal C_r\vee\mathcal N$
satisfies the identity $x^ry^m=x^{r+1}y^m$. Recall that the variety $\mathcal
C_m$ is generated by a monoid. Substituting 1 for $y$ in this identity, we
obtain that $\mathcal C_m$ satisfies the identity $x^r=x^{r+1}$. Therefore
$\mathcal C_m\subseteq\mathcal C_r$ contradicting the unequality $r<m$.

Thus we have proved that the set of varieties $\{\mathcal C_m\mid m\ge0\}$ is
definable by the formula $\mathtt{All\text{-}C_m}(x)$. Now Lemma \ref
{chain set} successfully applies with the conclusion that the variety
$\mathcal C_m$ is definable for each $m$.
\end{proof}

\begin{proposition}
\label{x^m=0,xy=yx}
For every natural number $m$, the variety $\mathcal D_m$ is definable.
\end{proposition}

\begin{proof}
Every commutative semigroup variety either coincides with $\mathcal{COM}$ or
is periodic. Thus the formula
$$\mathtt{Per}(x)\;\rightleftharpoons\;(\exists y)\,(x<y)$$
defines the set of all periodic commutative varieties. In particular, if
$\mathcal X$ is a commutative variety such that the sentence $\mathtt{Per}
(\mathcal X)$ is true then the variety $\Nil(\mathcal X)$ there exists. Put
$$\mathtt{Nil\text{-}part}(x,y)\;\rightleftharpoons\;\mathtt{Per}(x)\,\&\,y
\le x\,\&\,\mathtt{Nil}(y)\,\&\,(\forall z)\,\bigl(z\le x\,\&\,\mathtt{Nil}
(z)\longrightarrow z\le y\bigr)\ldotp$$
Clearly, if $\mathcal X$ and $\mathcal Y$ are commutative semigroup varieties
then the sentence $\mathtt{Nil\text{-}part}(\mathcal{X,Y})$ is true if and
only if $\mathcal X$ is periodic and $\mathcal{Y=\Nil(X)}$. Let $\mathtt
{C_m}$ be the formula defining the variety $\mathcal C_m$. The variety
$\mathcal D_m$ is defined by the formula
$$\mathtt{D_m}(x)\;\rightleftharpoons\;(\exists y)\,\bigl(\mathtt{C_m}(y)\,
\&\,\mathtt{Nil\text{-}part}(y,x)\bigr)$$
because $\mathcal D_m=\Nil(\mathcal C_m)$.
\end{proof}

If $\mathcal X$ is a commutative nil-variety of semigroups then we denote by
$\ZR(\mathcal X)$ the least 0-reduced in \textbf{Com} variety that contains
$\mathcal X$. Clearly, the variety $\ZR(\mathcal X)$ is given within
$\mathcal{COM}$ by all 0-reduced identities that hold in $\mathcal X$. If $u$
is a word and $x$ is a letter then $c(u)$ denotes the set of all letters
occurring in $u$, while $\ell_x(u)$ stands for the number of occurrences of
$x$ in $u$.

\begin{lemma}
\label{A_n key}
Let $m$ and $n$ be natural numbers with $m>2$ and $n>1$. The following are
equivalent:
\begin{itemize}
\item[\textup{(i)}]$\Nil(\mathcal A_n\vee\mathcal{X)=\ZR(X)}$ for any variety
$\mathcal{X\subseteq D}_m$;
\item[\textup{(ii)}]$n\ge m-1$.
\end{itemize}
\end{lemma}

\begin{proof}
(i)$\longrightarrow$(ii) Suppose that $n<m-1$. Let $\mathcal X$ be the
subvariety of $\mathcal D_m$ given within $\mathcal D_m$ by the identity
\begin{equation}
\label{A_n key ident}
x^{n+1}y=xy^{n+1}\ldotp
\end{equation}
Since $n+1<m$, the variety $\mathcal X$ is not 0-reduced in \textbf{Com}. Note
that $\mathcal{X\subseteq\Nil(A}_n\vee\mathcal X)$ because $\mathcal X$ is a
nil-variety. The identity \eqref{A_n key ident} holds in the variety
$\mathcal A_n\vee\mathcal X$, and therefore in the variety $\Nil(\mathcal A_n
\vee\mathcal X)$. But the latter variety does not satisfy the identity
$x^{n+1}y=0$ because this identity fails in $\mathcal X$. We see that the
variety $\Nil(\mathcal A_n\vee\mathcal X)$ is not 0-reduced in \textbf{Com}.
Since the variety $\ZR(\mathcal X)$ is 0-reduced in \textbf{Com}, we are
done.

(ii)$\longrightarrow$(i) Let $n\ge m-1$ and $\mathcal{X\subseteq D}_m$. One
can verify that $\mathcal A_n\vee\mathcal{X=A}_n\vee\ZR(\mathcal X)$. Note
that this equality immediately follows from \cite[Lemma 2.5]
{Shaprynskii-dnel} whenever $n\ge m$. We reproduce here the corresponding
arguments for the sake of completeness. It suffices to check that $\mathcal
A_n\vee\mathcal{\ZR(X)\subseteq A}_n\vee\mathcal X$ because the opposite
inclusion is evident. Suppose that the variety $\mathcal A_n\vee\mathcal X$
satisfies an identity $u=v$. We need to prove that this identity holds in
$\mathcal A_n\vee\ZR(\mathcal X)$. Since $u=v$ holds in $\mathcal A_n$, we
have $\ell_x(u)\equiv\ell_x(v)(\text{mod}\,n)$ for any letter $x$. If $\ell_x
(u)=\ell_x(v)$ for all letters $x$ then $u=v$ holds in $\mathcal A_n\vee\ZR
(\mathcal X)$ because this variety is commutative. Therefore we may assume
that $\ell_x(u)\ne\ell_x(v)$ for some letter $x$. Then either $\ell_x(u)\ge
n$ or $\ell_x(v)\ge n$. We may assume without any loss that $\ell_x(u)\ge n$.

Suppose that $n\ge m$. Then the identity $u=0$ holds in the variety $\mathcal
D_m$, whence it holds in $\mathcal X$. This implies that $v=0$ holds in
$\mathcal X$ too. Therefore the variety $\ZR(\mathcal X)$ satisfies the
identities $u=0=v$. Since the identity $u=v$ holds in $\mathcal A_n$, it
holds in $\mathcal A_n\vee\ZR(\mathcal X)$, and we are done.

It remains to consider the case $n=m-1$. Let $x$ be a letter with $x\in c(u)
\cup c(v)$ and $\ell_x(u)\ne\ell_x(v)$. If either $\ell_x(u)\ge m$ or $\ell_x
(v)\ge m$, we go to the situation considered in the previous paragraph. Let
now $\ell_x(u),\ell_x(v)<m$. Since $\ell_x(u)\ge n=m-1$, $\ell_x(u)\equiv
\ell_x(v)(\text{mod}\,n)$ and $\ell_x(u)\ne\ell_x(v)$, we have $\ell_x(u)=n=
m-1$ and $\ell_x(v)=0$. The latter equality means that $x\notin c(v)$.
Substituting 0 for $x$ in $u=v$, we obtain that the variety $\mathcal X$
satisfies the identity $v=0$. We go to the situation considered in the
previous paragraph again.

We have proved that $\mathcal A_n\vee\mathcal{X=A}_n\vee\ZR(\mathcal X)$.
Therefore $\ZR(\mathcal{X)\subseteq\Nil(A}_n\vee\mathcal X)$. If the variety
$\mathcal X$ satisfies an identity $u=0$ then $u^{n+1}=u$ holds in $\mathcal
A_n\vee\mathcal X$. This readily implies that $u=0$ in $\Nil(\mathcal A_n\vee
\mathcal X)$. Hence $\Nil(\mathcal A_n\vee\mathcal{X)\subseteq\ZR(X)}$. Thus
$\Nil(\mathcal A_n\vee\mathcal{X)=\ZR(X)}$.
\end{proof}

Now we are well prepared to prove the following

\begin{proposition}
\label{A_n}
An arbitrary Abelian periodic group variety is definable.
\end{proposition}

\begin{proof}
Abelian periodic group varieties are exhausted by the trivial variety and the
varieties $\mathcal A_n$ with $n>1$. The trivial variety is obviously
definable. For brevity, put
$$\mathtt{ZR}(x,y)\;\rightleftharpoons\;\mathtt{0\text{-}red}(y)\,\&\,x\le
y\,\&\,(\forall z)\,\bigl(\mathtt{0\text{-}red}(z)\,\&\,x\le z\longrightarrow
y\le z\bigr)\ldotp$$
The sentence $\mathtt{ZR}(\mathcal{X,Y})$ is true if and only if $\mathcal{Y=
\ZR(X)}$. Let $m$ be a natural number with $m>2$. In view of Lemma \ref
{A_n key}, the formula
$$\mathtt{A_{\ge m-1}}(x)\;\rightleftharpoons\;\mathtt{Gr}(x)\,\&\,(\forall
y,z,t)\,\bigl(\mathtt{D_m}(y)\,\&\,z\le y\,\&\,\mathtt{Nil\text{-}part}(x\vee
z,t)\longrightarrow\mathtt{ZR}(z,t)\bigr)$$
defines the set of varieties $\{\mathcal A_n\mid n\ge m-1\}$. Therefore the
formula
$$\mathtt{A_n}(x)\;\rightleftharpoons\;\mathtt{A_{\ge n}}(x)\,\&\,\neg\mathtt
{A_{\ge n+1}}(x)$$
defines the variety $\mathcal A_n$.
\end{proof}

It was proved in \cite{Kisielewicz-04} that each Abelian group variety is
definable in the lattice \textbf{Com}. However this paper contain no explicit
first-order formula defining any given Abelian periodic group variety.

Now we are ready to achieve the goal of this note.

\begin{theorem}
\label{comm monoid def}
A semigroup variety generated by a commutative monoid is definable.
\end{theorem}

\begin{proof}
Let $\mathcal V$ be a variety generated by some commutative monoid. According
to Lemma \ref{comm monoid}, $\mathcal{V=A}_n\vee\mathcal C_m$ for some $n\ge
1$ and $m\ge0$. It is easy to check that the parameters $n$ and $m$ in this
decomposition are defined uniquely. Therefore the formula
$$(\exists y,z)\,\bigl(\mathtt{A_n}(y)\,\&\,\mathtt{C_m}(z)\,\&\,x=y\vee z
\bigr)$$
defines the variety $\mathcal V$ (we assume here that $\mathtt{A_1}$ is the
evident formula defining the variety $\mathcal A_1=\mathcal T$).
\end{proof}

\medskip

\textbf{Acknowledgement.} The author thanks Dr.\ Olga Sapir for many
stimulating discussions.

\end{document}